\documentclass{article}
\usepackage{latexsym}
\usepackage{amssymb}
\usepackage{amsmath}
\newtheorem{theorem}{Theorem}[section]
\newtheorem{lemma}[theorem]{Lemma}
\newtheorem{corollary}[theorem]{Corollary}
\begin{document}
\title{\bf A Generalisation of Maximal (k,b)-Linear-Free Sets of Integers}
\author{\Large Nguyen Quang Minh \\University of Cambridge \\ Trinity College, Cambridge CB2 1TQ, England\\
{\tt quangminh112358@gmail.com}}
\date{}
\maketitle
\begin{abstract}
Fix integers $k,b,q$ with $k\ge 2,b\ge 0,q\ge 2$. Define the function $p$ to be: $p(x)=kx+b$. We call a set $S$ of integers \emph{$(k,b,q)$-linear-free} if $x\in S$ implies $p^i (x)\notin S$ for all $i=1,2,\dots,q-1$, where $p^i (x)=p\left(p^{i-1}(x)\right)$ and $p^0 (x)=x$. Such a set $S$ is maximal in $[n]:=\{1,2,\dots,n\}$ if $S\cup \{t\},\forall t\in [n]\backslash S$ is not $(k,b,q)$-linear-free. Let $M_{k,b,q} (n)$ be the set of all maximal $(k,b,q)$-linear-free subsets of $[n]$, and define $g_{k,b,q} (n)=\min_{S\in M_{k,b,q} (n)}|S|$ and $f_{k,b,q} (n)=\max_{S\in M_{k,b,q} (n)}|S|$. In this paper, formulae for $g_{k,b,q} (n)$ and $f_{k,b,q} (n)$ are proposed. Also, it is proven that there is at least one maximal $(k,b,q)$-linear-free subset of $[n]$ with exactly $x$ elements for any integer $x$ between $g_{k,b,q} (n)$ and $f_{k,b,q} (n)$, inclusively.
\end{abstract}

\section{Introduction}
A set of integers $S$ is said to be \emph{$k$-multiple-free} if $x\neq ky$ for all $x,y\in S$. Let $g_k (n)=\max\{|A|:A\subseteq [n]:=\{1,2,\dots,n\}  \textrm{ is $k$-multiple-free}\}$. E.T.H Wang \cite{b3}, when working on \emph{double-free} sets, showed that 
$$g_2(n) = \left\lceil{\frac{n}{2}}\right\rceil + g_2\left(\left\lfloor{\frac{n}{4}}\right\rfloor\right),$$
with $g_2 (0)=0$, and an asymptotic formula for $g_2 (n)$ is given by $g_2(n) = \frac{2}{3}n + O(\log n)$. In a similar vein, Leung and Wei \cite{b5} proved that for every integer $r>1$, $g_r (n)=\frac{r}{r+1}n+O(\log n).$

Instead of $k$-multiple-free subsets, one can ask more general questions about a subset $S$ of $[n]$ such that: for any two elements $x$ and $y$ of $S,$ $\frac{y}{x}$ does not belong to $C,$ where $C$ is a fixed set of rational numbers. For example, when $C=\{2,3\},$ such a subset $S$ is called \emph{strongly triple-free}, which is a subject of great interest in, for example, \cite{b4} and \cite{b8}. In fact, the genesis of this paper is the case when $C=\{k,k^2,\dots,k^q\},$ where $k$ and $q$ are positive integers.

On the other hand, You, Yu, and Liu \cite{b6} studied an alternative, natural generalisation of multiple-free sets, which are translation-free, or linear-free sets. In particular, a set $X$ of integers is called \emph{$(k,b)$-linear-free} if $x\in X$ implies $kx+b\notin X$. Obviously, when $b=0,$ $X$ is $k$-multiple-free. Define $g_{k,b}(n)=\min\{|X|:X\subseteq [n]\text{ is a maximal $(k,b)$-linear-free subset}\}$ and $f_{k,b}(n)=\max\{|X|:X\subseteq [n]\text{ is a maximal $(k,b)$-linear-free subset}\}$.  You, Yu, and Liu proposed explicit formulae for the two functions and showed that there is a maximal $(k,b)$-linear-free subset of $[n]$ with exactly $x$ elements for any integer $x$ between the minimum and maximum possible orders. Moreover, Liu and Zhou \cite{b1} also dealt with the same generalisation, albeit having a different approach.

This paper extends the aforementioned results by combining those two concepts. Henceforth, $k,b,$ and $q$ will denote integers with $k\ge 2, b\ge 0,$ and $q\ge 2$. Define the function $p$ to be: $p(x)=kx+b$. A set $S$ of integers is said to be \emph{$(k,b,q)$-linear-free} if $x\in S$ implies $p^i (x)\notin S,\forall i=1,2,\dots,q-1$, where $p^i (x)=p\left(p^{i-1}(x)\right)$ and $p^0 (x)=x$. Such a set $S$ is \emph{maximal} in $[n]$ if $S\cup \{t\}$ is not $(k,b,q)$-linear-free for any $t\in[n]\backslash S$. Let $M_{k,b,q} (n)$ be the set of all maximal $(k,b,q)$-linear-free subsets of $[n]$, and define $g_{k,b,q} (n)=\min_{S\in M_{k,b,q} (n)}|S|$ and $f_{k,b,q} (n)=\max_{S\in M_{k,b,q} (n)}|S|$. In this paper, formulae for $g_{k,b,q} (n)$ and $f_{k,b,q} (n)$ are derived. Also, it is proven that there is at least one maximal $(k,b,q)$-linear-free subset of $[n]$ of cardinality $x$ for any integer $x$ between $g_{k,b,q} (n)$ and $f_{k,b,q} (n)$, inclusively.

Notice that when $q=2$ we get back the notion of $(k,b)$-linear-free subsets: $f_{k,b,2} (n)=f_{k,b} (n)$ and $g_{k,b,2} (n)=g_{k,b} (n)$. 

There is yet another classical way to look at the topic. For a fixed $r \times s$ integer matrix $\textbf{A}$, let us call a subset $S_\textbf{A}(n)\subseteq [n]$ \emph{\textbf{A}-missing} if for every vector $\textbf{x}=\begin{pmatrix} x_1   \\ x_2\\ \vdots\\x_s \end{pmatrix}$ with $x_i\in S_\textbf{A}$ for all $i,$ the column vector $\textbf{Ax}$ has none of its entries being equal to $0.$ With this definition, take $\textbf{A}=\begin{pmatrix} k&  -1 \end{pmatrix}$ and $\textbf{B}=\begin{pmatrix} 2&  -1&0 \\ 3&  0&-1  \end{pmatrix},$ then \textbf{A}-missing subsets and \textbf{B}-missing subsets are, respectively, $k$-multiple-free subsets and strongly triple-free subsets. Naturally, we can further fix a $r\times 1$ matrix \textbf{b} and look at the entries of $\textbf{Ax}+\textbf{b}$. We specify the $(q-1) \times q$ matrix $\textbf{A}$ and the $(q-1) \times 1$ matrix $\textbf{b}$ to be
$$\textbf{A}=\begin{pmatrix}
    k & -1 & 0&\dots & 0  \\
    k^2 & 0 & -1&\dots & 0 \\
    \vdots & \vdots &\vdots& \ddots & \vdots \\
    k^{q-1} & 0 &0 & \dots & -1 
  \end{pmatrix}\hspace{0.5cm}\text{and}\hspace{0.5cm}\textbf{b}=\begin{pmatrix} p^1(0)   \\ p^2(0)\\ \vdots\\p^{q-1}(0) \end{pmatrix}.$$
Interpreted in terms of matrices, a $(k,b,q)$-linear-free subset $S_\textbf{A,b}(n)\subseteq [n]$ is equivalently a subset which satisfies: for every vector $\textbf{x}=\begin{pmatrix} x_1   \\ x_2\\ \vdots\\x_q \end{pmatrix}$ with $x_i\in S_\textbf{A,b}$ for all $i,$ the column vector $\textbf{Ax}+\textbf{b}$ has none of its entries being equal to $0.$ For instance, the case of $(k,b)$-linear-free subsets corresponds to $\textbf{A}=\begin{pmatrix}    k & -1 \\  \end{pmatrix}$ and $\textbf{b}=\begin{pmatrix}    b \\  \end{pmatrix},$ and none of the vectors $\textbf{x}=\begin{pmatrix} x_1   \\ x_2\end{pmatrix}$ with $x_1,x_2\in S_\textbf{A,b}$ is a solution to the equation $\textbf{Ax}+\textbf{b}=\textbf{0}.$

\section{Main Results}
To simplify expressions, we denote $\left\langle k^i\right\rangle=(k^i-1)/(k-1).$
\begin{theorem}
The recurrence relation for $f_{k,b,q} (n)$ is given by
$$f_{k, b, q}(n) = \left\lceil{\frac{n(k - 1) + b}{k}}\right\rceil + f_{k,b,q}\left(\left\lfloor{\frac{n-b\left\langle k^q \right\rangle}{k^q}}\right\rfloor\right),\hspace{0.8cm}\forall n > b$$
with $f_{k,b,q} (n)=0,\forall n<0$ and $f_{k,b,q} (n)=n,\forall n\in [0,b]$. Similarly,
$$g_{k, b, q}(n) = \left\lceil{\frac{n(k - 1) + b}{k}}\right\rceil + g_{k,b,q}\left(\left\lfloor{ \frac{n-b\left\langle k^{2q-1}\right\rangle}{k^{2q-1}}}\right\rfloor\right),\hspace{0.3cm}\forall n > b$$
with $g_{k,b,q} (n)=0,\forall n<0$ and $g_{k,b,q} (n)=n,\forall n\in[0,b]$. 
\end{theorem}
\noindent\emph{Proof.} Assume that $F$ is a maximal $(k,b,q)$-linear-free subset of $[n]$ with $f_{k,b,q} (n)$ elements. Fix $k$ and $q$, and partition $[n]$ into
    $$X = \left\{1, 2, 3, ..., \left\lfloor{\frac{n-b\left\langle k^q\right\rangle}{k^q}}\right\rfloor\right\}\hspace{0.3cm}\text{and}$$ $$Y = \left\{\left\lfloor{\frac{n-b\left\langle k^q\right\rangle}{k^q}}\right\rfloor+ 1,\left\lfloor{\frac{n-b\left\langle k^q\right\rangle}{k^q}}\right\rfloor+2, ..., n\right\}.$$
We will prove that, from the set $Y,F$ cannot have more than $\left\lceil{\frac{n(k - 1)+b}{k}}\right\rceil$ elements. Indeed, for each $i = \left\lfloor{\frac{n-b}{k}}\right\rfloor + 1, \left\lfloor{\frac{n-b}{k}}\right\rfloor + 2, ..., n$, let $h(i)$ be the greatest non-negative integer such that $p^{-h(i) } (i)\in [n]$, where  $p^{-i} (x)$ is the inverse function of $p^i (x)$. Define
$$Y_i=\left\{p^{-t} (i),0\le t\le \min\{ h(i),q-1\}\right\}.$$
The union of $Y_i$, $i = \left\lfloor{\frac{n-b}{k}}\right\rfloor + 1, \left\lfloor{\frac{n-b}{k}}\right\rfloor + 2,\dots,n$ $\Big($there are
$n- \left\lfloor{\frac{n-b}{k}}\right\rfloor=\left\lceil{\frac{n(k - 1)+b}{k}}\right\rceil$ such sets$\Big)$ covers $Y.$ This is because, for every $c\in Y$, there exists $j\in[1,q]$ satisfying $\left\lfloor{\frac{n-b\left\langle k^j\right\rangle}{k^j}}\right\rfloor + 1 \le c \le \left\lfloor{\frac{n-b\left\langle k^{j-1}\right\rangle}{k^{j - 1}}}\right\rfloor$. As $\left\lfloor{\frac{n-b}{k}}\right\rfloor + 1 \le p^{j - 1}(c) \le n,$ this guarantees $c\in Y_{p^{j-1} (c) }.$

Because $F$ cannot include more than one element in each $Y_i$, and can have at most $f_{k,b,q}\left(\left\lfloor{\frac{n-b\left\langle k^q\right\rangle}{k^q}}\right\rfloor\right)$ elements in $X$, so 
$$f_{k, b, q}(n) \le \left\lceil{\frac{n(k - 1) + b}{k}}\right\rceil + f_{k,b,q}\left(\left\lfloor{\frac{n-b\left\langle k^q\right\rangle}{k^q}}\right\rfloor\right).$$
On the other hand, let $A_1$ be a maximal $(k,b,q)$-linear-free subset of $X$ and $A_2 = \left\{\left\lfloor{\frac{n-b}{k}}\right\rfloor + 1, \left\lfloor{\frac{n-b}{k}}\right\rfloor + 2, ..., n\right\}$. Then $A=A_1\cup A_2$ is a maximal $(k,b,q)$-linear-free subset of $[n]$, hence 
$$f_{k, b, q}(n) \ge \left\lceil{\frac{n(k - 1) + b}{k}}\right\rceil + f_{k,b,q}\left(\left\lfloor{\frac{n-b\left\langle k^q\right\rangle}{k^q}}\right\rfloor\right).$$
The  upper and lower bounds complete the proof for the first part of Theorem 2.1. 

The second recurrence relation regarding $g_{k,b,q} (n)$ can be proven in the same fashion.
$\hfill\square$

Based on the proof of Theorem 2.1, a straightforward algorithm to construct a maximal $(k,b,q)$-linear-free subset $F$ of $[n]$ with $f_{k,b,q} (n)$ elements is proposed: first, check the largest number in $[n]$, which is $n$. As $p^i (n)\notin F,\forall i=1,2,\dots,q-1$, we add $n$ to $F$ (initially, $F$ is empty). Next, assume that we have checked until $x>1$. We then proceed to check $x-1$. If $p^i (x-1)\notin F,\forall i=1,2,\dots,q-1$, then $x-1$ is added to $F$. Otherwise, we proceed to check $x-2$. This algorithm terminates after number $1$ is checked. 

The next theorem provides formulae to determine $f_{k, b, q}(n)$ and $g_{k, b, q}(n).$
\begin{theorem}
For every positive integer $n,n\ge b\ge 1$:
$$f_{k, b, q}(n) = f_{k, b, q}\left(\left\lfloor{ \frac{n-b\left\langle k^{q(m + 1)} \right\rangle}{k^{q(m + 1)}}}\right\rfloor\right)+ \sum_{i = 0}^m\left\lceil{\frac{k - 1}{k}\left\lfloor{\frac{n-b\left\langle k^{qi} \right\rangle}{k^{qi}}}\right\rfloor} + \frac{b}{k}\right\rceil,$$
where $m = \left\lfloor{\frac{\log_k \left((n(k-1)+b)/bk\right)}{q}}\right\rfloor$. Similarly,
$$g_{k, b, q}(n) = g_{k, b, q}\left(\left\lfloor{ \frac{n-b\left\langle k^{(2q-1)(m + 1)}\right \rangle}{k^{(2q-1)(m + 1)}}}\right\rfloor\right)+ \sum_{i = 0}^m\left\lceil{\frac{k - 1}{k}\left\lfloor{\frac{n-b\left\langle k^{(2q-1)i}\right\rangle}{k^{(2q-1)i}}}\right\rfloor} + \frac{b}{k}\right\rceil,$$
where $m = \left\lfloor{\frac{\log_k \left((n(k-1)+b)/bk\right)}{2q-1}}\right\rfloor$. 
\end{theorem}
It should be remarked that, by Theorem 2.1,  $f_{k, b, q}\left(\left\lfloor{ \frac{n-b\left\langle k^{q(m + 1)} \right\rangle}{k^{q(m + 1)}}}\right\rfloor\right)$
is either $0$ or $\left\lfloor{ \frac{n-b\left\langle k^{q(m + 1)} \right\rangle}{k^{q(m + 1)}}}\right\rfloor,$ depending on whether $\left\lfloor{ \frac{n-b\left\langle k^{q(m + 1)} \right\rangle}{k^{q(m + 1)}}}\right\rfloor$ is negative or not, and $g_{k, b, q}\left(\left\lfloor{ \frac{n-b\left\langle k^{(2q-1)(m + 1)}\right \rangle}{k^{(2q-1)(m + 1)}}}\right\rfloor\right)$ is determined correspondingly.

\noindent\emph{Proof.} For $f_{k,b,q} (n)$, this is obtained by repeatedly applying Theorem 2.1 until $\left\lfloor{ \frac{n-b\left\langle k^{qi}\right \rangle}{k^{qi}}}\right\rfloor\le b$, i.e., $i\ge\left\lfloor{\frac{\log_k \left((n(k-1)+b)/bk\right)}{q}}\right\rfloor.$ Also, notice that
$$\left\lfloor{\frac{\left\lfloor{\frac{n-b\left\langle k^{qi} \right\rangle}{k^{qi}}}\right\rfloor-b\left\langle k^{q} \right\rangle}{k^{q}} }\right\rfloor=\left\lfloor{\frac{n-b\left\langle k^{q(i+1)} \right\rangle}{k^{q(i+1)}}}\right\rfloor.$$
The same argument can be applied to deduce the formula for $g_{k,b,q} (n)$.
$\hfill\square$

In Theorem 2.2, when $q=2$, formulae for $f_{k,b} (n)$ and $g_{k,b} (n)$ are obtained, which are, however, different from those of You, Yu, and Liu \cite{b6}. Particularly, You, Yu, and Liu suggested that
$$f_{k,b}(n)=\sum_{i=1}^{n(1)}|N_i|\left\lceil{\frac{i+1}{2}}\right\rceil \hspace{0.5cm}\text{and}\hspace{0.5cm}g_{k,b}(n)=\sum_{i=1}^{n(1)}|N_i|\left\lceil{\frac{i+1}{3}}\right\rceil,$$
where $n(1)=\left\lfloor{\log_k\frac{n+b/(k-1)}{1+b/(k-1)}}\right\rfloor$ and 
$$  |N_i| = \begin{cases} \left\lfloor{\frac{n-b\left\langle k^{i} \right\rangle}{k^{i}}}\right\rfloor-2\left\lfloor{\frac{n-b\left\langle k^{i+1} \right\rangle}{k^{i+1}}}\right\rfloor+\left\lfloor{ \frac{n-b\left\langle k^{i+2}\right \rangle}{k^{i+2}}}\right\rfloor & \text{for } i< n(1), \\ \left\lfloor{\frac{n-b\left\langle k^{i} \right\rangle}{k^{i}}}\right\rfloor &\text{for } i = n(1). \end{cases}  $$
These are essentially different from Theorem 2.2 as the two have different strategies: our approach utilises recurrence relations while You, Yu, and Liu partition $[n]$ from the beginning. Furthermore, it seems that their approach can be suitably modified to accommodate for the more general cases when $q\ge 3.$

These somewhat complicated formulae are illustrated in the following example:

\noindent\textbf{Example 1.} To determine $f_{2,1,3} (2021)$ and $g_{2,1,3} (2021)$, Theorem 2.2 gives:
$$f_{2, 1, 3}(2021) = f_{2, 1, 3}\left(\left\lfloor{\frac{2021}{8^{4}} - \frac{8^4-1}{8^4}}\right\rfloor\right)+ \sum_{i = 0}^3\left\lceil{\frac{1}{2}\left\lfloor{\frac{2021}{8^{i}} - \frac{8^{i} - 1}{8^{i}}}\right\rfloor} + \frac{1}{2}\right\rceil=1155$$
and
$$g_{2, 1, 3}(2021) = g_{2, 1, 3}\left(\left\lfloor{\frac{2021}{32^{2}} - \frac{32^2-1}{32^2}}\right\rfloor\right)+ \sum_{i = 0}^1\left\lceil{\frac{1}{2}\left\lfloor{\frac{2021}{32^{i}} - \frac{32^{i} - 1}{32^{i}}}\right\rfloor} + \frac{1}{2}\right\rceil=1043$$

\begin{corollary}
$$f_{k,0, q}(n) = \frac{k^{q - 1}(k - 1)}{k^q - 1}n + O(\log n).$$

\end{corollary}
This is based on 
$$f_{k,0, q}(n) = \sum_{i \ge 0}\left\lceil{\frac{k - 1}{k}\left\lfloor{\frac{n}{k^{qi}}}\right\rfloor}\right\rceil,$$
which is a direct corollary of Theorem 2.2. In fact, Corollary 2.3 can be extended by generalising the method of Wakeham and Wood \cite{b2}. Particularly, for positive integers $a,b$ whose greatest common divisor is $g$ and $a<b$, one has 
$$f_{b/a,0,q}(n) = \frac{b^{q - 1}(b - g)}{b^q - g^q}n + O(\log n).$$

Next, still in the case $b=0,$ if $n$ is a perfect power of $k$ then a much neater formula for $f_{k,0,q}(n)$ is as follows:
\begin{theorem}
$$f_{k,0, q}(k^n) = \left\lfloor{\frac{k^{q - 1}(k^n - 1)(k - 1)}{k^q - 1}}\right\rfloor + 1.$$
\end{theorem}
\noindent\emph{Proof.}
We will induct on $n$. First, let $0\le n<q$. The theorem is true for $n=0$ as $f(1)=1$. If $1\le n<q$, Theorem 2.1 suggests
$$f_{k,0, q}(k^n) = \left\lceil{\frac{k - 1}{k}k^n}\right\rceil + f_{k, 0,q}(0) = k^{n - 1}(k - 1).$$
We can then prove that, for $0\le n<q$:
$$f_{k, 0,q}(k^n) = \left\lfloor{\frac{k^{q - 1}(k^n - 1)(k - 1)}{k^q - 1}}\right\rfloor + 1.$$
Lastly, the inductive step from $n$ to $n+q$ can be done as follows
$$\left\lfloor{\frac{k^{q - 1}(k^n - 1)(k - 1)}{k^q - 1}}\right\rfloor + 1 = \left\lfloor{\frac{k^{q - 1}(k^{n - q} - 1)(k - 1)}{k^q - 1}}\right\rfloor + 1 + (k - 1)k^{n - 1}.$$
This completes the proof. $\hfill\square$

There is an alternative to determine $f_{k,0,q} (n)$. Partition $[n]$ into geometric progressions of the form $\{t,kt,k^2 t,\dots,k^j t\}$, where $t$ is not divisible by $k$ and $1\le t\le n$. Each such subset’s cardinality is $\left\lfloor{\log_k(n / t)}\right\rfloor + 1.$ Within each subset, if its cardinality is $M$, then at most $\left\lceil{M/q}\right\rceil$ elements can belong to the same maximal $(k,0,q)$-linear-free subset of $[n]$. Therefore,
$$f_{k,0, q}(n) = \sum_{\substack{{1\le t\le n} \\ {k\nmid t}}}\left\lceil{\frac{\left\lfloor{\log_k(n / t)}\right\rfloor + 1}{q}}\right\rceil.$$
Replacing $n$ with $k^{n-1}$, we obtain a result in The American Mathematical Monthly \cite{b7}:
\begin{quote}
    Let $k,q$ and $n$ be positive integers with $k\ge 2$, and let $P$ be the set of all positive integers less than $k^n$ that are not divisible by $k$. Prove
$$\sum_{p \in P}\left\lceil{\frac{\left\lfloor{n - \log_k p}\right\rfloor}{q}}\right\rceil = \left\lfloor{\frac{k^{q - 1}(k^{n - 1} - 1)(k - 1)}{k^q - 1}}\right\rfloor + 1.$$
\end{quote}

Next, inspired by theorem 2 in \cite{b6}, we investigate the possible cardinalities of maximal $(k,b,q)$-linear-free subsets. The following lemma is crucial to the proof of Theorem 2.6.

\begin{lemma}
A set $S$ is said to be \emph{$q$-distanced} if for all $x,y\in S,x\neq y$, then $|x-y|\ge q$. Let the set $H_q (n)$ include every subset $S$ of $[n]$ such that: $S$ is $q$-distanced but $S\cup \{t\}, \forall t\in [n]\backslash S$ is not. Then, there exists a set $A\in H_q (n)$ whose cardinality is $x$ if and only if $\left\lceil{\frac{n}{q}}\right\rceil \ge x \ge \left\lceil{\frac{n}{2q - 1}}\right\rceil$.
\end{lemma}
\noindent\emph{Proof.}
Let $a$ be an integer in $[0,n]$. From $a$ numbers from $1$ to $a$, pick all the numbers which are congruent to $1$ modulo $q$ to be elements of $S$ (initially, $S$ is empty).  For the other $n-a$ numbers, pick all the numbers which are congruent to $a+q$ modulo $2q-1$ to be in $S$. Then $S\in H_q (n)$ and its cardinality is $\left\lceil{\frac{a}{q}}\right\rceil+\left\lceil{\frac{n-a}{2q-1}}\right\rceil.$ It suffices to prove that there exists an integer $a\in [0,n]$ such that $\left\lceil{\frac{a}{q}}\right\rceil+\left\lceil{\frac{n-a}{2q-1}}\right\rceil=x$ if and only if $\left\lceil{\frac{n}{q}}\right\rceil \ge x \ge \left\lceil{\frac{n}{2q - 1}}\right\rceil$.

Indeed, consider $a$ to be of the form $(2q-1)k$, where $k$ is a non-negative integer. Let the function $s:\left[0,\left\lfloor{\frac{n}{2q-1}}\right\rfloor\right]\longrightarrow \mathbb{N}$ be
$$s(k)=\left\lceil{\frac{(2q-1)k}{q}}\right\rceil+\left\lceil{\frac{n-(2q-1)k}{2q-1}}\right\rceil=k+\left\lceil{\frac{-k}{q}}\right\rceil+\left\lceil{\frac{n}{2q-1}}\right\rceil.$$
Then 
$$s(k+1)-s(k)=\left(1+\left\lceil{\frac{-k-1}{q}}\right\rceil-\left\lceil{\frac{-k}{q}}\right\rceil\right)\in\{0,1\}.$$
Hence, $s$ is a non-decreasing function and each of its jump is exactly $1$ unit. Furthermore, $s(0)=\left\lceil{\frac{n}{2q-1}}\right\rceil,\left\lceil{\frac{n}{q}}\right\rceil+\left\lceil{\frac{n-n}{2q-1}}\right\rceil=\left\lceil{\frac{n}{q}}\right\rceil$ and
$$s\left(\left\lfloor{\frac{n}{2q-1}}\right\rfloor\right)=\left\lfloor{\frac{n}{2q-1}}\right\rfloor+\left\lceil{\frac{-\left\lfloor{\frac{n}{2q-1}}\right\rfloor}{q}}\right\rceil+\left\lceil{\frac{n}{2q-1}}\right\rceil.$$
Hence,
$$s\left(\left\lfloor{\frac{n}{2q-1}}\right\rfloor\right)\ge \frac{n}{2q-1}-1-\frac{n}{(2q-1)q}+\frac{n}{2q-1}=\frac{n}{q}-1$$
$$\Longrightarrow s\left(\left\lfloor{\frac{n}{2q-1}}\right\rfloor\right)\ge\left\lceil{\frac{n}{q}}\right\rceil-1.$$
This shows that $\left\lceil{\frac{a}{q}}\right\rceil+\left\lceil{\frac{n-a}{2q-1}}\right\rceil$ can take any integral values between $\left\lceil{\frac{n}{q}}\right\rceil$ and $\left\lceil{\frac{n}{2q-1}}\right\rceil$, inclusively, completing the proof.
 $\hfill\blacksquare$

This leads to 
\begin{theorem}
For any integer $x$ in $\left[g_{k,b,q} (n),f_{k,b,q} (n)\right]$, there is at least one maximal $(k,b,q)$-linear-free subset of $[n]$ containing exactly $x$ elements.
\end{theorem}
\noindent\emph{Proof.}
Equipped with Lemma 2.5, one can proceed by making obvious modifications to the proof in \cite{b6}.   $\hfill\square$

Lastly, we determine values of $n$ for which the lower bound and upper bound of the cardinalities of $(k,b,q)$-linear-free subsets of $[n]$ are, in fact, identical.  
\begin{theorem}
$g_{k,b,q} (n)=f_{k,b,q}(n)$ if and only if $n<p^q (1)$.
\end{theorem}
\noindent\emph{Proof.} One can show that
$$  f_{k, b, q}(n + 1) = \begin{cases} f_{k, b, q}(n) & \text{if } q \nmid h(n + 1), \\ f_{k, b, q}(n) + 1 &\text{if } q \mid h(n + 1). \end{cases}  $$
Similarly, 
$$  g_{k, b, q}(n + 1) = \begin{cases} g_{k, b, q}(n) & \text{if } 2q-1 \nmid h(n + 1), \\ g_{k, b, q}(n) + 1 &\text{if } 2q-1 \mid h(n + 1). \end{cases}  $$
Hence, when $n<p^q (1)$, for any $x\in[n], q\mid h(x)$ if and only if $h(x)=0$, and so $(2q-1)\mid h(x)$. This implies that $f_{k,b,q} (n)=g_{k,b,q} (n)$ due to the recurrence relations. On the other hand, when $n\ge p^q (1)$, the rate of growth of $f_{k,b,q} (n)$ is evidently faster than that of $g_{k,b,q} (n)$, and thus, they cannot be equal.
$\hfill\square$

\section{Further Research}
For future research, it will be interesting to investigate the problem in a multi-dimensional space. For instance, one can start from the following generalisation for the double-free subsets:

Let $m$ and $n_1,n_2,\dots,n_m$ be positive integers. Define $T=\{(x_1,x_2,\dots,x_m ):1\le x_i\le n_i,\forall i=\overline{1,m}\}$. Let $g(n_1,\dots,n_m )$ be the maximum number of lattice points one can choose from $T$ such that if $(x_1,\dots,x_m)$ is chosen, then $(2x_1,\dots,2x_m )$ is not. Determine $g(n_1,\dots,n_m )$.

It can be proven that
$$g(n_1,\dots,n_m)=\prod_{i=1}^m n_i-\prod_{i=1}^m\left\lfloor\frac{n_i}{2}\right\rfloor+g\left(\left\lfloor\frac{n_1}{4}\right\rfloor,\dots,\left\lfloor\frac{n_m}{4}\right\rfloor\right).$$
Hence,
$$g(n_1,\dots,n_m)=\sum\limits_{i\ge0}(-1)^i\prod_{j=1}^m\left\lfloor\frac{n_j}{2^i}\right\rfloor.$$

It is unclear whether Lemma 2.5 can be extended in this case, even to the Cartesian plane, which is critical to generalise Theorem 2.6.


\begin{thebibliography}{99}
\bibitem{b1}	Bolian Liu and Zhou Bo, \emph{A generalization of maximal k-multiple-free sets of integers}. Ars Combinatoria 52(1999), 285-288.
\bibitem{b2}	D. Wakeham and D.R. Wood, \emph{On multiplicative Sidon sets}. Integers: Annual Volume 2013, (Ed. Bruce Landman). Berlin, Boston: De Gruyter, 2014, 392-401. https://doi.org/10.1515/9783110298161.392
\bibitem{b3}	Edward T. H. Wang, \emph{On double-free sets of integers}. Ars Combinatoria, 28(1989), 97–100.
\bibitem{b4}	F. Chung, P. Erdős, and R. Graham, \emph{On Sparse Sets Hitting Linear Forms}. In Number Theory for the Millennium, vol. 1, Proc. 2000 Urbana Conf. (Ed. M. A. Bennett, B. C. Berndt, N. Boston, H. G. Diamond, A. J. Hildebrand, and W. Philipp). A K Peters, Natick, MA, 2002, 257-272.
\bibitem{b5}	J. Y-T Leung and W-D Wei, \emph{Maximal k-multiple-free sets of integers}. Ars Combinatoria, 38(1994), 113–117.
\bibitem{b6}	Lihua You, Gexin Yu, Bolian Liu, \emph{On maximal (k,b)-linear-free sets of integers and its spectrum}. Australasian Journal of Combinatorics, 23(2001), 211-215.  
\bibitem{b7}	Nguyen Quang Minh, \emph{Problems 12252}. The American Mathematical Monthly, 128:5(2021), 467.
\bibitem{b8}	R. Graham, J. Spencer, and H. Witsenhausen, \emph{On Extremal Density Theorems for Linear Forms}. In Number Theory and Algebra (Ed. H. Zassenhaus). Academic Press, New York, 1977, 103-109.





\end{thebibliography}
\end{document}